%% file: unitarygroups.tex
\documentclass[12pt]{article}

\input{common.tex}

\begin{document}

\title{Automorphic representations with prescribed ramification for unitary groups}
\author{William Conley\thanks{Partially funded by EPSRC grant EP/G001480/1}}
\date{September 2011}
\maketitle

\begin{abstract}
Let $F$ be a totally real number field, $n$ a prime integer, and $G$ a unitary group of rank $n$ defined over $F$ that is compact at every infinite place. We prove an asymptotic formula for the number of cuspidal automorphic representations of $G$ whose factors at finitely many places are prescribed up to inertia. The results and the methods used are a generalization to this setting of those used by Weinstein for $\textrm{GL}_2$. 
\end{abstract}

\section{Introduction}
\label{sec:introduction}

In the representation theory of real and $\p$-adic groups, much use is made of studying representations of a group $G$ via their restrictions to certain compact subgroups. Here we use the same strategy for an adelic group $G$, in order to count the multiplicity of representations with certain ramification behavior in the automorphic spectrum of $G$. To specify such ramification behavior precisely, we will use the theory of types for $\p$-adic groups, and will define a corresponding notion of global types for $G$. Our first goal is thus a purely local result, and the main theorem will follow as an application of this. We begin by describing the local work. 

\subsection{Local Theory}
\label{sec:localintro}

Let $F$ be a nonarchimedean local field, let $G = \GL_n(F)$, and let $\pi$ be an irreducible supercuspidal representation of $G$. The inertial equivalence class of $\pi$ then consists of all isomorphism classes of unramified twists of $\pi$. Given such an inertial equivalence class $\fr{s}$, a type for $\fr{s}$ (or an $\fr{s}$-type) is a pair $(J, \lambda)$ consisting of a compact open subgroup $J$ of $G$ and an irreducible representation $\lambda$ of $J$, such that for any irreducible representation $\pi$ of $G$, $\pi \in \fr{s}$ if and only if the restriction of $\pi$ to $J$ has a subspace isomorphic to $\lambda$. Types for all supercuspidal representations of $\GL_n(F)$ were first constructed, for the case when $n$ is prime, by Carayol in \cite{Carayol:Cuspidals}. This was later generalized to all $n$ by Bushnell and Kutzko in \cite{BK:Orange}. Although we will only deal with the former case here, we will universally use the notation and terminology of the latter, as it has become somewhat standard. 

For the global application considered here, it seems easiest to work with a fixed compact subgroup of $G$. For this reason, the types described above are inconvenient, due to the fact that the subgroup $J$ varies with $\fr{s}$. Thus we use a variant of the theory, which appears to have first been considered for $\GL_2$ in \cite{BM:Modular}, and was further studied in \cite{Paskunas}. Let $K = \GL_n(\roi_F)$, let $\fr{s}$ be an inertial equivalence class of supercuspidal representations of $G$, and let $(J, \lambda)$ be a maximal simple type for $\fr{s}$ (see \cite{BK:Orange} or section~\ref{sec:typesintro} below). By conjugating $J$ and $\lambda$ by an element of $g$ as necessary, we may assume that $J \subset K$, and we let $\tau = \Ind_J^K(\lambda)$. Then $\tau$ is irreducible (see \cite{Paskunas} for details), and thus by Frobenius reciprocity $(K, \tau)$ is an $\fr{s}$-type. The main result of \cite{Paskunas} was that, for a supercuspidal inertial equivalence class $\fr{s}$, such a representation $\tau$ is the unique $\fr{s}$-type defined on $K$. Our first goal here will be to establish, for a large class of elements $g \in K$, some bounds on $\Tr(\tau(g))$ as $\tau$ varies over all such types. This is done in Theorem~\ref{th:tracebound} below, and is the key local result that we will need to establish our main theorem. 

\subsection{Global Theory}
\label{sec:globalintro}

We now move to the global setting. For a number of reasons, we have chosen to focus here on automorphic representations defined on a certain class of unitary groups, however it should be possible to carry out a similar program for a large class of other groups. Here $F$ will denote a totally real number field and $E$ a totally imaginary quadratic extension of $F$. Let $n$ be prime as before, and let $M$ be a central simple algebra of dimension $n^2$ over $E$. Denote by $x \mapsto x^*$ an involution of the second kind of $M$, \ie an $F$-algebra anti-automorphism of $M$ of order $2$ whose restriction to $E$ (the center of $M$) is the non-trivial element of $\Gal(E/F)$. Let $G$ be the unitary group defined (over $F$) by $M$ and $^*$. Explicitly, this is given by 
\[ G(R) = \left\{ g \in M \otimes_F R \suchthat g g^* = 1 \right\} \quad \text{for every $F$-algebra $R$.} \]
In all that follows, we will fix a choice of $M$ and $^*$ for which $G(F_v)$ is compact for each infinite place $v$ of $F$. For each such $v$, we fix an isomorphism 
\[ \iota_v: G(F_v) \to \ro{U}(n) . \]
(The relevant details of the representation theory of $\ro{U}(n)$ will be reviewed in section~\ref{sec:archimedean} below.) Let $S$ be the set of places of $F$ which split in $E$. For each $v \in S$, we will fix an isomorphism 
\[ \iota_v: G(F_v) \to \GL_n(F_v) . \]
For each infinite place $v$ of $F$, we let $K_v = G(F_v)$, for each $v \in S$, we let $K_v = \iota_v^{-1}(\GL_n(\roi_{F_v}))$, and for each finite place $v \notin S$, we let $K_v$ be any fixed maximal compact subgroup of $G(F_v)$. 

Let $Z$ be the center of $G$, which is the unitary group of rank $1$ defined over $F$ using the extension $E / F$. This is given explicitly by 
\[ Z(R) = \left\{ x \in E \otimes_F R \suchthat x x^* = 1 \right\} \quad \text{for every $F$-algebra $R$.} \]
Note that $Z(F)$ is just $E^1 = \left\{ x \in E \suchthat N_{E/F}(x) = 1 \right\}$. 

Let $\AA$ be the ring of adeles of $F$, and $\AA_E$ that of $E$. Recall that there is a natural embedding of $\AA$ into $\AA_E$ and a norm map $N_{E/F}: \AA_E \to \AA$. With this notation, the adelic points of the center of $G$ are given by 
\[ Z(\AA) = \left\{ x \in \units{\AA_E} \suchthat N_{E/F}(x) = 1 \right\} . \]
Let $K$ be the subgroup of $G(\AA)$ given by 
\[ K = \prod_v K_v , \]
and let $Z_0$ be the center of $K$. Since $G$ was chosen to be compact at all the infinite places of $F$, $K$ (resp. $Z_0$) is actually a maximal compact open subgroup of $G(\AA)$ (resp. $Z(\AA)$). Note that the subgroup of rational points of $Z_0$ is just the group $\roi_E^1$ of units of norm $1$ in $\roi_E$, which is simply the finite group $\mu_E$ of roots of unity in $E$. 

For a character $\omega$ of $Z(\AA)$ that is trivial on $Z(F)$, we let $\ca{A}(\lquot{G(F)}{G(\AA)}, \omega)$ be the space of automorphic forms on $G(\AA)$ with central character $\omega$. Since $G$ is compact at each infinite place, this is simply the space of smooth complex-valued functions on $G(\AA)$ that are invariant under left translation by elements of $G(F)$, and transform by $\omega$ under left translation by elements of $Z(\AA)$. As usual, the group $G(\AA)$ acts on $\ca{A}(\lquot{G(F)}{G(\AA)}, \omega)$ by right translation, and the resulting representation decomposes as a direct sum of irreducible subrepresentations, in which each isomorphism class occurs with finite multiplicity. An irreducible representation of $G(\AA)$ occurring in $\ca{A}(\lquot{G(F)}{G(\AA)}, \omega)$ has central character $\omega$, and these representations, as $\omega$ ranges over all characters of $\lquot{Z(F)}{Z(\AA)}$, are the automorphic representations of $G(\AA)$. For an automorphic representation $\pi$ with central character $\omega$, we will write $m(\pi)$ for its multiplicity in $\ca{A}(\lquot{G(F)}{G(\AA)}, \omega)$. 

Since $\lquot{Z(F)}{Z(\AA)}$ is compact, its spectrum (the group of characters $\omega$ considered above) is discrete. Thus we may consider the space of \emph{all} automorphic forms on $G(\AA)$, which is simply a discrete direct sum: 
\[ \ca{A}(\lquot{G(F)}{G(\AA)}) = \bigoplus_{\omega} \ca{A}(\lquot{G(F)}{G(\AA)}, \omega) . \]
Note that this is simply the space of smooth functions on $G(\AA)$ that are left $G(F)$-invariant. Clearly each automorphic representation $\pi$ of $G(\AA)$ occurs in $\ca{A}(\lquot{G(F)}{G(\AA)})$ with multiplicity $m(\pi)$. 

If $\pi$ is an automorphic representation of $G(\AA)$, its central character $\omega_{\pi}$ is by definition trivial on $Z(F)$, so the restriction of $\omega_{\pi}$ to $Z_0$ must be trivial on $Z_0 \inter Z(F) = \roi_E^1$. This motivates the last point in the following definition: 

\begin{definition} \label{def:globaltype}
A \define{global type for $G$} is an irreducible representation 
$\tau = \bigotimes_v \tau_v$ of $K$ satisfying the following: 
\begin{enumerate}
    \item For each place $v$ of $F$, $\tau_v$ is an irreducible representation of $K_v$. 
    \item For all finite places $v \notin S$ and almost all $v \in S$, $\tau_v = 1$. 
    \item For all $v \in S$ for which $\tau_v$ is not $1$-dimensional, $\tau_v = \tau'_v \circ \iota_v$, where $\tau'_v$ is the type of a supercuspidal inertial equivalence class for $\GL_n(F_v)$. 
    \item If $\omega_v$ is the central character of $\tau_v$ for each place $v$, then the character $\omega_{\tau} = \prod \omega_v$ of $Z_0$ is trivial on $\roi_E^1$. 
\end{enumerate}
\end{definition}

Now let $\pi = \bigotimes' \pi_v$ be an automorphic representation of $G(\AA)$ for which $\pi_v$ is either supercuspidal or a twist of an unramified representation at each finite place $v$ of $F$, and is unramified for each $v \notin S$. For each place $v$, we can define a representation $\tau_v(\pi)$ of $K_v$, which we may call the type of $\pi$ at $v$, as follows: 
\begin{enumerate}
    \item If $v$ is an infinite place of $F$, let $\tau_v(\pi) = \pi_v$. 
    \item If $v$ is a finite place not in $S$, let $\tau_v(\pi)$ be the trivial representation of $K_v$. 
    \item If $v \in S$, regard $\pi_v$ as a representation of $\GL_n(F_v)$ via $\iota_v$, and let $\tau_v(\pi) = \tau \circ \iota_v$, where $\tau$ is the unique representation of $\GL_n(\roi_{F_v})$ that is a type for $\pi_v$. (If $\pi_v$ is a twist of an unramified representation, $\tau_v(\pi)$ will be the corresponding twist of the trivial representation of $K_v$.) 
\end{enumerate}
Now let $\tau(\pi) = \bigotimes_v \tau_v(\pi)$. Then $\tau(\pi)$ is a global type for $G$, and it is clearly the unique global type that occurs in the restriction of $\pi$ to $K$, and occurs in $\pi$ with multiplicity $1$. We will call this the global type corresponding to $\pi$, or more succinctly, the type of $\pi$. 

Since a global type $\tau$ is an irreducible representation of the compact group $K$, it admits a unitary central character $\omega_{\tau}$. By definition $\omega_{\tau}$ is a character of the group $Z_0$, trivial on the finite subgroup $\roi_E^1$. If $\pi$ is an automorphic representation of $G(\AA)$ of type $\tau$, then its central character $\omega_{\pi}$ must be an extension of $\omega_{\tau}$ from $Z_0$ to $Z(\AA)$ that is trivial on $E^1$. Note that there are a finite number of such extensions, determined by the characters of the finite group $\lquot{Z(F) Z_0}{Z(\AA)}$ (which is easily seen to be isomorphic to a subgroup of the ideal class group of $F$). 

There is an obvious notion of twisting a global type by a character of $K$, which is compatible with the twisting of automorphic representations by characters of $G(\AA)$. Specifically, let $\theta_v$ be a character of $K_v$ for each place $v$, such that $\theta_v = 1$ for almost all finite places $v$ and all $v \notin S$, and such that 
\[ \theta^n |_{\roi_E^1} = 1 , \]
where $\theta = \prod \theta_v$. Then $\theta \otimes \tau$ will be a global type as well. We will use the notation $\theta \tau$ for the twist of $\tau$ by $\theta$ so defined. 

Such a character $\theta$ of $K$ can always be extended to a unitary character $\chi = \prod \chi_v$ of $G(\AA)$, for which $\chi_v$ will be unramified for almost all $v \in S$ and all finite $v \notin S$, and for which $\chi^n |_{E^1} = 1$. Conversely, given such a character $\chi$ of $G(\AA)$, its restriction $\theta$ to $K$ will satisfy all the requirements of the previous paragraph. If $\pi$ is an automorphic representation of $G(\AA)$ of type $\tau$, then we can twist $\pi$ by the character $\chi$ to obtain an automorphic representation $\chi \pi$, and clearly it will have type $\theta \tau$. Thus for the purposes of counting automorphic representations of a given type, it will suffice to deal with global types only up to twisting. 

\subsection{Main Theorem}
\label{sec:maintheoremstatement}

We may now state our main theorem. For any global type $\tau$, let $m(\tau)$ be the multiplicity of $\tau$ in the restriction of $\ca{A}(\lquot{G(F)}{G(\AA)})$ to $K$, and let $\ca{R}(\tau)$ denote the set of distinct isomorphism classes of automorphic representations of $G(\AA)$ of type $\tau$. We intend to count the number of automorphic representations of type $\tau$ by computing $m(\tau)$. Note that, since the type $\tau$ of an automorphic representation $\pi$ always occurs in $\pi$ with multiplicity one, we can only conclude that $m(\tau) = \card{\ca{R}(\tau)}$ for all global types $\tau$ if the multiplicity one theorem holds for $G$. We will not assume this here, so we cannot conclude that $m(\tau)$ is equal to the number of distinct automorphic representations of type $\tau$. But in any case $m(\tau)$ is the sum of the multiplicities of these automorphic representations: 
\[ m(\tau) = \sum_{\pi \in \ca{R}(\tau)} m(\pi) , \]
and in general, $m(\tau) \ge \card{\ca{R}(\tau)}$. And we can certainly conclude that $\ca{R}(\tau) \neq \emptyset \iff m(\tau) > 0$. In other words, there exist automorphic representations of type $\tau$ if and only if $m(\tau)$ is not zero. 

For any global type $\tau$, let $S(\tau)$ be the set of finite places $v$ for which $\dim(\tau_v) > 1$. Note that both $\dim(\tau)$ and the set $S(\tau)$ are invariant under twisting. To deal with the infinite places, let $\fr{h}_{\RR}^*$ denote the space of weight vectors for the Lie group $\rm{U}(n)$. (See section~\ref{sec:archimedean} below for details.) For a global type $\tau$ and an infinite place $v$ of $F$, we will denote by $\lambda_v(\tau) \in \fr{h}_{\RR}^*$ the highest weight vector of $\tau_v$ (viewed as a representation of $\ro{U}(n)$ via the isomorphism $\iota_v$). The Weyl dimension formula then gives the dimension of $\tau_v$ as a polynomial function of $\lambda_v(\tau)$. We will refer to this polynomial as the Weyl polynomial of $\ro{U}(n)$. Our main theorem is now 

\begin{theorem} \label{th:globaltheorem}
There exist constants $C_1$ and $C_2$, with $C_1 > 0$, and for each infinite place $v$ of $F$ a polynomial $P_v$ on $\fr{h}_{\RR}^*$, all depending only on the group $G$, such that for all global types $\tau$, 
\[ m(\tau) \ge C_1 \dim(\tau) - C_2 \cdot n^{\card{S(\tau)}} \cdot \prod_{v \divs \infty} P_v(\lambda_v(\tau)) . \]
Each of the polynomials $P_v$ has degree strictly less than that of the Weyl polynomial of $\ro{U}(n)$. 
\end{theorem}

Both of the constants and the polynomials $P_v$ appearing in this theorem can be computed explicitly for any particular example of a group $G$, so that the theorem could yield quite explicit results. But we note that in any event, we have the following immediate corollary. 

\begin{corollary}
For all but a finite number of twist classes of global types $\tau$ of $G$, there exist automorphic representations of $G(\AA)$ of type $\tau$. 
\end{corollary}
\begin{proof}
For $v \in S$, the smallest possible dimension of a supercuspidal type defined on $K_v \cong \GL_n(\roi_{F_v})$ is $(q_v-1)(q_v^2-1) \dotsm (q_v^{n-1}-1)$, where $q_v$ is the cardinality of the residue field of $F_v$. (See section~\ref{sec:tracebound} below for details.) Obviously this is greater than $n$ for almost all $v$. Let $P$ denote the Weyl polynomial for $\ro{U}(n)$. Then for any infinite place $v$ of $F$, $\dim(\tau_v) = P(\lambda_v(\tau))$, so 
\[ \dim(\tau) \ge \prod_{v \in S(\tau)} \left( (q_v-1) \dotsm (q_v^{n-1}-1) \right) \cdot \prod_{v \divs \infty} P(\lambda_v(\tau)) \]
for all global types $\tau$. Thus, assuming the notation of the Theorem, if we enumerate the twist classes of global types $\tau$, it is clear that 
\[ \frac{\dim(\tau)}{n^{\card{S(\tau)}} \cdot \prod\limits_{v \divs \infty} P_v(\lambda_v(\tau))} \]
grows without bound. The result now follows. 
\end{proof}

As previously noted, our main theorem and the methods used to prove it are a direct generalization of those in \cite{Weinstein}. It is also worth noting that Shin has recently announced, in \cite{Shin}, a result along similar lines. While Shin's result applies to a much more general class of groups and automorphic representations, the asymptotic formula he derives, specialized to this case, is quite different from ours. The reason for this is that we are counting the raw number of automorphic representations of a given type, whereas his formula estimates the total dimension of certain isotypic subspaces within such representations. 

There are a number of directions in which this result could be generalized. The restriction that the global unitary group be compact at all infinite places greatly simplifies the global argument used here, but it is likely that an application of the trace formula to global types like the ones constructed here would yield similar results for a much greater class of groups. In the local setting, one obvious improvement would be to remove the restriction that $n$ be prime. Another would be to use types for $p$-adic unitary groups at the non-split places, rather than requiring that global types be trivial there (and thus restricting the result to automorphic representations that are unramified at these places). Types for the supercuspidal representations of $p$-adic unitary groups have been constructed in \cite{Stevens:Exhaustion}. Finally, it should be possible to treat more cases than just representations that are supercuspidal or twists of an unramified principal series, again using a more general class of types than the ones used here. This was done for $\GL_2$ by Weinstein, and only required treating the Steinberg representations as a special case. But for $n > 2$, trying to use local inertial types defined on a maximal compact subgroup becomes more complicated. The author has achieved some results in this area, which will appear in a forthcoming paper. 

\subsection*{Acknowledgements}

I must first and foremost thank Jared Weinstein, on whose original work this paper is based. Great thanks are also due to Don Blasius for his tireless guidance, and for looking over an earlier draft of this work. Without either of them, this paper would never have come to be. I would also like to thank Shaun Stevens for pointing out a mistake in an earlier draft of this work, and for many other helpful suggestions.

\section{Maximal simple types when $n$ is prime}
\label{sec:tametypes}

In this section and the next, we will focus on the local theory at the nonarchimedean places, so we return to the notation of section~\ref{sec:localintro}. Namely, let $F$ be a nonarchimedean local field, with ring of integers $\roi_F$, prime ideal $\p_F$, and residue field $\resfld_F = \rquot{\roi_F}{\p_F}$ of cardinality $q$. Throughout this entire section, we will assume that $n$ is prime. Though the types defined below were first constructed by Carayol in \cite{Carayol:Cuspidals}, our description of them follows \cite{BK:Orange} exactly, as does our basic notation and terminology. 

\subsection{Definition of types}
\label{sec:typesintro}

As we will occasionally make use of explicit matrix computations, we fix $V = F^n$, and we fix a basis of $V$ so that we may identify $A = \End_F(V)$ with $M_n(F)$ and $G = \units{A}$ with $\GL_n(F)$. We also fix a choice of additive character $\psi$ of $F$ of level zero. The construction of types in this setting breaks up naturally into three cases: 

\begin{description}
\item[The depth zero case] \hfill \\
Let $\tau$ be the twist by a character of $\units{\roi_F}$ of the inflation to $K$ of a cuspidal irreducible representation of $\GL_n(\resfld_F)$. Then $\tau$ is a supercuspidal type for $G$. 

\item[The unramified case] \hfill \\
Let $\fr{A} = M_n(\roi_F)$, a hereditary $\roi_F$-order in $A$, and let $\fr{P} = \p M_n(\roi_F)$, the Jacobson radical of $\fr{A}$. Let $\beta \in A \smallsetminus \fr{A}$ such that $E = F[\beta]$ is an unramified field extension of $F$ of degree $n$, such that $\units{E}$ normalizes $\fr{A}$, and such that $\beta$ is \define{minimal} over $F$ (see \cite{BK:Orange}*{1.4.14}). Let $m$ be the unique (positive) integer such that $\beta \in \fr{P}^{-m} \smallsetminus \fr{P}^{-m+1}$. Define a character $\psi_{\beta}$ of the group $1 + \fr{P}^{\ceil{\frac{m+1}{2}}}$ by 
\[ \psi_{\beta}(x) = \psi(\Tr_{A/F}(\beta(x-1))) . \]
Define groups $H^1$, $J^1$, and $J$ as follows: 
\begin{align*}
    H^1 &= (1 + \p_E) (1 + \fr{P}^{\ceil{\frac{m+1}{2}}}) ,                 \\
    J^1 &= (1 + \p_E) (1 + \fr{P}^{\floor{\frac{m+1}{2}}}) \text{, and }    \\
    J   &= \units{\roi_E} (1 + \fr{P}^{\floor{\frac{m+1}{2}}}) . 
\end{align*}
Let $\theta$ be any extension of $\psi_{\beta}$ to $H^1$. There is a unique irreducible representation $\eta$ of $J^1$ whose restriction to $H^1$ contains $\theta$. Let $\lambda$ be any extension of $\eta$ from $J^1$ to $J$. The pair $(J, \lambda)$ is now a special case of maximal simple type, in the language of \cite{BK:Orange}. Finally, let $\tau = (\chi \circ \det) \otimes \Ind_J^K(\lambda)$, for any character $\chi$ of $\units{\roi_F}$. Then $\tau$ is a supercuspidal type for $G$. 

\item[The ramified case] \hfill \\
Let $\fr{A}$ be the subring of matrices in  $M_n(\roi_F)$ that are upper triangular modulo $\p$, which is also a hereditary $\roi_F$-order in $A$. Let $\fr{P}$ again be the Jacobson radical of $\fr{A}$, which is the ideal of matrices whose reductions modulo $\p$ are nilpotent upper triangular. Much like before, let $\beta \in A \smallsetminus \fr{A}$ such that $E = F[\beta]$ is a totally ramified field extension of $F$ of degree $n$, such that $\units{E}$ normalizes $\fr{A}$, and such that $\beta$ is minimal over $F$. Define $m$, the groups $H^1$, $J^1$, and $J$, the characters $\psi_{\beta}$ and $\theta$, and the representations $\eta$, $\lambda$, and $\tau$ exactly as in the unramified case above. Once again, $(J, \lambda)$ is a special case of maximal simple type, and $\tau$ is likewise a supercuspidal type for $G$. 
\end{description}

The main result of \cite{Paskunas} is that every supercuspidal type for $G$ that is defined on $K$ is one of the representations $\tau$ described above. Furthermore, for any irreducible supercuspidal representation $\pi$ of $G$, the restriction of $\pi$ to $K$ contains one and only one such type, and that type occurs with multiplicity one in $\pi$. 

\subsection{A preliminary trace bound}
\label{sec:glaubermanlemma}

Our first task is to more carefully analyze the representation $\lambda$ in one particular case, namely when $E$ is unramified and $\lambda$ is \emph{not} $1$-dimensional (\ie when $m$ is even). The result that we derive here is probably well known to the experts, but the exact statement that we require does not seem to appear in the literature. At any rate, the details are quite technical, so we collect them here. It is likely that a very similar statement holds more generally, but the lemma below is sufficient for our needs. The proof of this lemma is very similar to others found in the literature (see for example~\cite{BH:Lifting2}*{4.1 - 4.2},~\cite{BH:Essentially1}*{4.1}), but adapted to the current setting. 

\begin{lemma} \label{lem:glauberman}
Let $\fr{A}$, $\beta$, and $m$ be as in the unramified case described above, and assume that $m$ is even. Let $H^1$, $J^1$, $J$, $\theta$, $\eta$, and $\lambda$ also be as above. Then 
\[ \abs{\Tr \lambda(a(1+x))} = 1 \]
for any $x \in \fr{P}^{\floor{\frac{m+1}{2}}}$ and any $a \in \units{\roi_E}$ whose reduction modulo $\p_E$ is not in $\units{\resfld_F}$. 
\end{lemma}
\begin{proof}
For convenience, let $k = \floor{\frac{m+1}{2}} = \frac{m}{2}$, so that 
\begin{align*}
    H^1 &= (1+\p_E)(1+\fr{P}^{k+1}) ,           \\
    J^1 &= (1+\p_E)(1+\fr{P}^k) , \text{ and }  \\
    J   &= \units{\roi_E}(1+\fr{P}^k). 
\end{align*}
Note that we have an exact sequence 
\begin{equation} \label{eq:splitexact}
    1 \to J^1 \to J \to \units{\resfld_E} \to 1 , 
\end{equation}
which in this case splits since $\units{\resfld_E} \cong \mu_E$, the group of roots of unity of order prime to $p$ in $E$. Thus $J = \units{\resfld_E} \ltimes J^1$, where the action of $\units{\resfld_E}$ on $J^1$ is by conjugation. 

Recall (from \cite{BK:Orange}*{3.3.1} for example) that $\theta$ is fixed under conjugation by $J$. Thus $\Ker\theta \normsubgp J$, so we let 
\[ \Bar{H^1} = \rquot{H^1}{\Ker\theta}, \quad \Bar{J^1} = \rquot{J^1}{\Ker\theta}, \quad \Bar{J} = \rquot{J}{\Ker\theta}, \]
and let $\Bar{\theta}$ (resp. $\Bar{\eta}$, $\Bar{\lambda}$) be the composition of $\theta$ (resp. $\eta$, $\lambda$) with the quotient map. Thus $\Bar{\eta}$ is the unique irreducible representation of $\Bar{J^1}$ whose restriction to $\Bar{H^1}$ contains $\Bar{\theta}$, and $\Bar{\lambda}$ is an extension of $\Bar{\eta}$ to $\Bar{J}$. 

Let $W = \rquot{\Bar{J^1}}{\Bar{H^1}} \cong \rquot{J^1}{H^1}$, and define $h_{\theta}: W \times W \to \units{\CC}$ by 
\[ (\Bar{x}, \Bar{y}) \mapsto \theta[x,y] . \]
By \cite{BK:Orange}*{3.4.1}, $h_{\theta}$ is a nondegenerate alternating bilinear form on the $\resfld_F$-vector space $W$, from which it follows that $\Bar{H^1}$ is the center of $\Bar{J^1}$ (and hence also of $\Bar{J}$). 

Although we will not need this result, note that in this setting 
\[ W \cong \rquot{(1+\fr{P}^k)}{(1+\p_E^k)(1+\fr{P}^{k+1})} \cong \rquot{\fr{P}^k}{\p_E^k + \fr{P}^{k+1}} \]
is a $\resfld_F$-vector space of dimension $n^2 - n$. Thus the representation $\lambda$ will have dimension $q^{\frac{n^2-n}{2}}$. 

The split exact sequence \eqref{eq:splitexact} reduces to 
\[ 1 \to \Bar{J^1} \to \Bar{J} \to \units{\resfld_E} \to 1 , \]
which still splits. We regard $\units{\resfld_E}$ as a group of automorphisms of $\Bar{J^1}$, acting by conjugation. Fix $a \in \units{\resfld_E} \smallsetminus \units{\resfld_F}$. Note that the commutator map $V \to V$ defined by $v \mapsto a^{-1} v a v^{-1}$ is an isomorphism. Thus if $g \in \Bar{J^1}$, we can choose $g_0 \in \Bar{J^1}$ such that $a^{-1} g_0 a g_0^{-1} = g h^{-1}$ for some $h \in \Bar{H^1}$, whence $g_0^{-1} (a g) g_0 = a h$. Thus every element of $\Bar{J}$ of the form $ag$, $g \in \Bar{J^1}$, is conjugate to an element of the form $ah$, with $h \in \Bar{H^1}$. So we will be finished if we can prove that $\abs{\Tr \Bar{\lambda}(ah)} = 1$ for all $h \in \Bar{H^1}$. 

Let $A = \langle a \rangle \subset \units{\resfld_E}$. Note that $\Bar{J^1}$ is a finite $p$-group (where $p$ is the characteristic of $\resfld_F$), so its order is relatively prime to that of $A$. Since $\Bar{\theta}$ is fixed by the action of $A$, the isomorphism class of $\Bar{\eta}$ is as well. Under these circumstances, in~\cite{Glauberman}, Glauberman gives a one-to-one correspondence between isomorphism classes of irreducible representations of $\Bar{J^1}$ fixed by $A$ and those of $\Bar{J^1}^A = \Bar{H^1}$. This correspondence maps $\Bar{\eta}$ to $\Bar{\theta}$ (by Theorem~5(d) of~\cite{Glauberman}, for example). By Theorem~2 of~\cite{Glauberman}, there exists a certain canonical extension of $\Bar{\eta}$ to $\Bar{J}$, and $\Bar{\lambda}$ is a twist of it by a uniquely determined character $\chi$ of $\units{\resfld_E}$. Thus by Theorem~3 of~\cite{Glauberman}, there exists a constant $\epsilon = \pm 1$ such that 
\[ \Tr \Bar{\lambda}(ah) = \epsilon\, \chi(a) \Bar{\theta}(h) \]
for all $h \in \Bar{H^1}$. (Note that the constant $\epsilon$ depends on $a$ and on $\eta$, but this need not concern us here.) The result now follows. 
\end{proof}

\section{A bound on the characters of types}
\label{sec:tracebound}

We now come to the first real result of this article. This is our main local result, and will provide the crucial ingredient in the proof of the main theorem. 

\begin{theorem} \label{th:tracebound}
Let $n$ be a prime integer, let $g \in K = \GL_n(\roi_F)$, and assume $g$ is not in the center of $K$. There exists a constant $C_g$ such that for all supercuspidal types $\tau$ defined on $K$, 
\[ \abs{\Tr(\tau(g))} \le C_g . \]
Let $\Bar{g} \in \GL_n(\resfld_F)$ be the reduction of $g$ modulo $\p_F$. Then if the characteristic polynomial of $\Bar{g}$ is irreducible, we may take $C_g = n$. Otherwise, if $\Bar{g}$ has at least two distinct eigenvalues, then we may take $C_g = 0$. 
\end{theorem}
\begin{proof}

We prove this theorem in three cases, corresponding to the three cases in the construction of types described in section~\ref{sec:typesintro}. Since twisting clearly has no bearing on the conclusions stated here, we may ignore the twisting by characters of $\units{\roi_F}$ that occurs as the last step of each of those cases. 

\subsection{The depth zero case}

Let $\tau$ be a depth zero type. Then after twisting by a character of $\units{\roi_F}$, we may assume that $\tau$ is merely the inflation to $K$ of an irreducible cuspidal representation of $\GL_n(\resfld_F)$. Since this group is finite, the first claim is clear for this case. 

The characters of the irreducible cuspidal representations of $\GL_n(\FF_q)$ were first computed in~\cite{Green:Characters}. We briefly recall the resulting formula, in a simplified form. Let $\FF_{q^n}$ denote the finite field of $q^n$ elements. A character $\theta$ of $\units{\FF_{q^n}}$ is called \define{regular} if its orbit under the action of $\Gal(\FF_{q^n} / \FF_q)$ has exactly $n$ elements, or in other words, if $\theta, \theta^q, \dotsc, \theta^{q^{n-1}}$ are all distinct. The cuspidal representations of $\GL_n(\FF_q)$ are in one-to-one correspondence with the orbits of these characters. Note that if $\theta$ is a regular character of $\units{\FF_{q^n}}$ and $k$ is any integer such that $k, kq, \dotsc, kq^{n-1}$ are distinct modulo $q^n - 1$, then $\theta^k$ will also be regular. Furthermore, if we fix a choice of regular character $\theta$ of $\units{\FF_{q^n}}$, the map $k \mapsto \theta^k$ induces a bijection between such integers and the set of regular characters. We will thus denote the corresponding irreducible cuspidal representation of $\GL_n(\FF_q)$ by $\tau_k$. Then \cite{Green:Characters}*{p. 431} gives the following: 
\begin{itemize}
    \item If the characteristic polynomial of $\Bar{g}$ is a power of a single irreducible polynomial $f$ of degree $d$, then 
    \[ \Tr(\tau_k(\Bar{g})) = (-1)^{n-1} \left( \prod_{i=1}^{r-1} (1-q^i) \right) \left( \sum_{\gamma} \theta^k(\gamma) \right) , \]
    where the sum is taken over the $d$ distinct roots $\gamma$ of $f$ in $\FF_{q^n}$, and $r$ is the number of Jordan blocks in the Jordan normal form of $\Bar{g}$ over $\FF_q$. 
    \item Otherwise, $\Tr(\tau_k(\Bar{g})) = 0$. 
\end{itemize}
Since we are assuming $n$ is prime, the first of these cases can happen only if the characteristic polynomial of $\Bar{g}$ is either irreducible or of the form $(x - \gamma)^n$ for some $\gamma \in \units{\FF_q}$. Thus, if $\Bar{g}$ has irreducible characteristic polynomial, we get $\abs{\Tr(\tau_k(\Bar{g}))} \le n$ as desired. And otherwise, if $\Bar{g}$ has at least two distinct eigenvalues, then $\Tr(\tau_k(\Bar{g})) = 0$. Furthermore, taking $\Bar{g} = 1$, we see that the dimension of $\tau_k$, and thus of any depth zero type, is $(q-1) (q^2-1) \dotsm (q^{n-1}-1)$. 

\subsection{The unramified case}

Let $\fr{A} = M_n(\roi_F)$, and temporarily assume all of the other notation from the unramified case of section~\ref{sec:typesintro}. Recall that in this case, $E = F[\beta]$ is an unramified extension of $F$ of degree $n$, where $\beta \in M_n(F)$ is minimal over $F$. (Note also that since $\beta \in M_n(F)$, we regard $E$ as being explicitly embedded in the $F$-algebra $M_n(F)$.) To simplify notation, we let $k = \floor{\frac{m+1}{2}}$, so that $J = \units{\roi_E}(1 + \fr{P}^k)$. Since the final step in the construction of the type $\tau$ in this case was induction from $J$ to $K$, we will naturally make use of the Frobenius formula: 
\begin{equation} \label{eq:frobenius}
    \Tr \Ind_J^K(\lambda)(g) = \sum_{\substack{x \in \rquot{K}{J} \\ x^{-1}gx \in J}} \Tr \lambda(x^{-1}gx) . 
\end{equation}
Note that in this formula, the condition $x^{-1}gx \in J$ is equivalent to $gxJ = xJ$, or in other words that the coset of $x$ in $\rquot{K}{J}$ is fixed under the natural left action of $K$. Thus to apply this formula, we will begin by defining a model of the coset space $\rquot{K}{J}$ that is equipped with the same left action of $K$, then determine the points fixed by the element $g$ in this space. 

As a starting point for our model of this coset space, note that there is a natural left action of $\GL_n(F)$ on $\PP^{n-1}(E)$. (If we think of elements of projective space as column vectors in homogeneous coordinates, this action is just given by matrix multiplication.) Note that the subset of elements whose homogeneous coordinates form a basis of $E$ over $F$ is stable under this action, and the group acts transitively on this set. Similarly, we have a natural left action of $K$ on $\PP^{n-1}(\roi_E)$, and we define $X$ to be the set of all points in $\PP^{n-1}(\roi_E)$ with homogeneous coordinates 
\[ [u_0 : \dotso : u_{n-1}] \]
such that $\{ u_0, \dotsc, u_{n-1} \}$ is an $\roi_F$-basis of $\roi_E$. It is clear that $K$ acts transitively on $X$, and that $\units{\roi_E}$ is the stabilizer of some point $x \in X$. A simple computation shows that the action of the normal subgroup $1 + \fr{P}^k$ induces the equivalence relation of congruence modulo $\p_E^k$ on the coordinates $u_i$ of points in $X$. If we let $X_k$ denote the quotient of $X$ under this equivalence (which we may think of as a subset of $\PP^{n-1}(\roi_E / \p_E^k)$), and let $x_k$ denote the class of $x$, then we have a $K$-equivariant bijection 
\[ \rquot{K}{J} \to X_k \]
defined by $aJ \mapsto a \cdot x_k$. 

Note that the choice of the point $x$ will depend on the embedding of $E$ into $M_n(F)$, and hence on the element $\beta \in M_n(F)$. Thus the actual bijection established here will vary for different types $(J, \lambda)$, even for different ones having the same value of $k$. However, the action of $K$ on $X_k$ is the same in all cases, and it will turn out that this will be all that matters for our purpose: since $\lambda$ has dimension either $1$ or $q^{\frac{n^2 - n}{2}}$ (depending on whether $m$ is odd or even, respectively), $\Tr(\lambda(g))$ is bounded by the latter value, so by the Frobenius formula, $\Tr(\tau(g))$ is bounded by this value times the number of fixed points of $g$ in $X^k$. Thus the first claim of the theorem will be proved for all unramified types once we can show that the number of fixed points of $g$ in $X_k$ is bounded as $k \to \infty$. Since this has nothing to do with the choice of a fixed type, we now forget about $J$ and $\lambda$ (and $\beta$, $m$, \etc) until near the end of this section, and work only with the sets $X_k$, for \emph{all} $k > 0$. 

Note that for each $k' < k$, we get a $K$-equivariant surjection $X_k \to X_{k'}$. (In fact, these form a projective system 
\[ X_1 \gets X_2 \gets \dotsb \]
of $K$-sets, and $X = \varprojlim X_k$, but we will not need this fact.) Clearly if $g$ has a fixed point in $X_k$, then the image of this point in $X_{k'}$ must be a fixed point of $g$ as well. 

\begin{lemma*}
Assume that $\Bar{g}$ is not a scalar. If the characteristic polynomial of $\Bar{g}$ is irreducible in $\resfld_F[x]$, then $g$ has at most $n$ fixed points in $X^k$ for all $k$. Otherwise, $g$ has no fixed points in $X^k$ for all $k$. 
\end{lemma*}
\begin{proof}
Note that $X^1$ is precisely the subset of $\PP^{n-1}(\resfld_E)$ consisting of points whose homogeneous coordinates form a basis of $\resfld_E$ over $\resfld_F$. Furthermore, this set carries the natural action of $\GL_n(\resfld_F) = \rquot{K}{1 + \p M_n(\roi_F)}$, and the action of $K$ on $X^1$ factors through this quotient. Thus the fixed points of $g$ in $X^1$ are precisely the fixed points of $\Bar{g}$ in $X^1$, which are just the one-dimensional spaces of eigenvectors of $\Bar{g}$ that coincide with points in $X^1$. Let $p(x) = x^n + a_{n-1}x^{n-1} + \dotsb + a_0$ be the characteristic polynomial of $g$ in $\roi_F[x]$, and let $\Bar{p}$ be its reduction modulo $\p$. Suppose that $\Bar{p} = p_1 p_2$, with $p_1$ and $p_2$ relatively prime in $\resfld_F[x]$. Then $\Bar{g}$ is similar to a block-diagonal matrix $\left( \begin{smallmatrix} g_1 & \\ & g_2 \end{smallmatrix} \right)$ such that the characteristic polynomial of $g_i$ is $p_i$ for $i = 1,2$. Clearly such a matrix cannot have a fixed point in $X^1$. 

So if $\Bar{g}$ has a fixed point in $X^1$, $\Bar{p}$ must be a power of a single irreducible polynomial. But since $n$ is prime, this means that either $\Bar{p}$ is irreducible, or $\Bar{p}(x) = (x - \alpha)^n$ for some $\alpha \in \units{\resfld_F}$. Assume the latter. Then $\Bar{g}$ is conjugate within $\GL_n(\resfld_F)$ to a matrix of the form $\alpha + h$, where $h$ is a nilpotent upper-triangular matrix. Since we are assuming that $\Bar{g}$ is not a scalar, $h \neq 0$. Again, it is clear that such a matrix cannot have a fixed point in $X^1$. This proves the second claim of the lemma. 

Now assume that $\Bar{p}$ is irreducible. Then by elementary linear algebra, there exists a $\Bar{g}$-cyclic vector $\Bar{v} \in \resfld_F^n$, \ie a vector for which 
\[ \{ \Bar{v}, \Bar{g} \Bar{v}, \dotsc, \Bar{g}^{n-1} \Bar{v} \} \]
is a basis of $\resfld_F^n$. It is not hard to see that any lift $v$ of $\Bar{v}$ from $\resfld_F^n$ to $\roi_F^n$ must be $g$-cyclic, and such a vector yields a basis of $F^n$ consisting of vectors in $\roi_F^n$. The matrix of $g$ with respect to this basis is the companion matrix 
\[ C_p = \begin{pmatrix}
    0      &       &       &-a_0        \\
    1      &\ddots &       &-a_1        \\
           &\ddots &0      &\vdots      \\
           &       &1      &-a_{n-1}
\end{pmatrix} . \]
Thus $g$ is conjugate within $K$ to $C_p$, so for the purposes of counting fixed points, we may assume $g = C_p$. Now with $g$ in this simplified form, an easy computation shows that the fixed points of $g$ in $X_k$ correspond precisely to the roots of the polynomial $p$ in $\units{(\rquot{\roi_E}{\p_E^k})}$. If $k = 1$, there are clearly at most $n$ of these. And since $\Bar{p}$ is irreducible, they are all distinct, so Hensel's lemma implies that there are at most $n$ such roots in $\units{(\rquot{\roi_E}{\p_E^k})}$ for all $k > 1$ as well. 
\end{proof}

The last claim of the theorem, in the unramified case, follows immediately from this lemma. To prove the second claim of the theorem in this case, we return to the context of the beginning of this section of the proof: let $(J, \lambda)$ be a maximal simple type with all its associated notation, and let $\tau = \Ind_J^K(\lambda)$. Now each of the at most $n$ fixed points of $g$ in $X_k$ corresponds to an $x \in \rquot{K}{J}$ such that $x^{-1} g x \in J$. Recalling that $J = \units{\roi_E}(1 + \fr{P}^k)$, we see that the reduction modulo $\p$ of $x^{-1} g x$ will be an element of $\units{(\rquot{\roi_E}{\p_E})}$. Since its minimal polynomial is irreducible of degree $n$ (because it is a conjugate of $\Bar{g}$) it must in fact be in $\units{\resfld_E} \smallsetminus \units{\resfld_F}$. Thus, if $m$ is even, Lemma~\ref{lem:glauberman} implies that $\abs{\Tr \lambda(x^{-1} g x)} \le 1$. On the other hand, if $m$ is odd, then $\lambda$ is one-dimensional, so the same is obviously true. Thus either way, the Frobenius formula implies 
\[ \abs{\Tr \tau(g)} \le n . \]

Finally, we deal with the first claim of the theorem: that the trace of $\tau(g)$ is bounded as $\tau$ runs over all unramified types of $K$. As remarked previously, this will be proved once we show that the number of fixed points of $g$ in $X^k$ is bounded as $k \to \infty$. The only case not covered by the lemma is when $\Bar{g}$ is a scalar. For this, we choose $\alpha \in \units{\roi_F}$ and $h \in \fr{P}^l = \p^l M_n(\roi_F)$ such that $g = \alpha + h$, and such that $l$ is maximal with respect to this decomposition. Let $\varpi$ be a uniformizer of $F$, and let $a = \varpi^{-l} h$. The assumption that $l$ is maximal is equivalent to assuming that $a \in \fr{A} \smallsetminus \fr{P}$ and $\Bar{a}$ is not scalar. 

Now let $\bo{u} = [u_0 : \dotso : u_{n-1}]$ represent a point in $X^k$ for some $k$. Then this point is fixed by $g$ if and only if 
\begin{equation} \label{eq:unramified_step1}
    \alpha \bo{u} + \varpi^l a \bo{u} = \gamma \bo{u} \pmod{\p_E^k} 
\end{equation}
for some $\gamma \in \units{\roi_E}$. Clearly if $k \le l$, then every point in $X^k$ yields a solution to this, taking $\gamma = \alpha \pmod{\p_E^k}$. Assume now that $k > l$. Since $a \in \fr{A} \smallsetminus \fr{P}$, this system of equations can have a solution only if we choose $\gamma = \alpha \pmod{\p_E^l}$, but $\gamma \neq \alpha \pmod{\p_E^{l+1}}$. Assuming this and letting $\gamma' = \varpi^{-l} (\gamma - \alpha) \in \units{\roi_E}$, \eqref{eq:unramified_step1} becomes 
\begin{equation} \label{eq:unramified_step2}
    a \bo{u} = \gamma' \bo{u} \pmod{\p_E^{k-l}} . 
\end{equation}
If $a \notin \GL_n(\roi_F)$, this has no solution, and hence $g$ has no fixed points in $X^k$ for any $k > l$. But if $a \in \GL_n(\roi_F)$, this says that $\bo{u}$ represents a fixed point of $g$ in $X^k$ if and only if $\bo{u}$ represents a fixed point of $a$ in $X^{k-l}$. Since $\Bar{a}$ is not scalar, the lemma implies that there are at most $n$ such points in $X^{k-l}$. Thus there are at most $n (\card{\resfld_E})^{ln}$ fixed points of $g$ in $X^k$. 

\subsection{The ramified case}

Now let $\fr{A}$ be the algebra of matrices in $M_n(\roi_F)$ that are upper triangular modulo $\p$, so that $\units{\fr{A}}$ is the Iwahori subgroup of $K$. As in the previous section of the proof, we temporarily assume all of the notation from the ramified case of section~\ref{sec:typesintro}. In particular, $E = F[\beta]$ is now a totally ramified extension of $F$ of degree $n$, where $\beta \in M_n(F)$ is minimal over $F$. Once again, let $k = \floor{\frac{m+1}{2}}$, so that $J = \units{\roi_E}(1 + \fr{P}^k)$. Let $\rho = \Ind_J^{\units{\fr{A}}}(\lambda)$, and let 
\[ \tau = \Ind_{\units{\fr{A}}}^K(\rho) = \Ind_J^K(\lambda) . \]
We will use the same strategy here as in the previous section of the proof, except that we will deal primarily with the induction to $\units{\fr{A}}$, which is now a proper subgroup of $K$. 

Let $\varpi$ be a uniformizer of $E$, and define $X \subset \PP^{n-1}(\roi_E)$ to be the set of all points with homogeneous coordinates 
\[ [ u_0 : u_1 \varpi : \dotso : u_{n-1} \varpi^{n-1} ], \quad u_i \in \units{\roi_E}. \]
(Note that this is equivalent to saying the coordinates form an $\roi_F$-basis of $\roi_E$, with strictly increasing $E$-valuations.) Again it is easy to see that $\units{\fr{A}}$ acts transitively on $X$, and that $\units{\roi_E}$ is the stabilizer of some point $x \in X$. A straightforward computation shows that in this case, the normal subgroup $1 + \fr{P}^k$ induces the equivalence relation of congruence modulo $\p_E^k$ on the units $u_i$ appearing in the coordinates of points in $X$: 
\begin{gather*}
    [ u_0 : u_1 \varpi : \dotso : u_{n-1} \varpi^{n-1} ] \sim [ u_0' : u_1' \varpi : \dotso : u_{n-1}' \varpi^{n-1} ] \\
    \text{if and only if} \\
    u_i = u_i' \pmod{\p_E^k} \text{ for each } i , 
\end{gather*}
or in other words, congruence modulo $\p_E^{k+i}$ on the $i$th coordinate, for each $i$. If we once again let $X_k$ denote the quotient of $X$ under this equivalence, and let $x_k$ denote the class of $x$, then $aJ \mapsto a \cdot x_k$ again defines an $\units{\fr{A}}$-equivariant bijection 
\[ \units{\fr{A}} / J \to X_k . \]
The same comments apply as before: the actual bijection given above will be different for subgroups $J$ coming from different types, but the action of $\units{\fr{A}}$ on the set $X_k$ will be the same regardless; and since the dimension of $\lambda$ is bounded by a fixed value, we may now forget all about the specific type, and deal only with counting fixed points of $g$ in the sets $X_k$, for all $k > 0$. Also as before, we have a projective system 
\[ X_1 \gets X_2 \gets \dotsb \]
of $\units{\fr{A}}$-sets, with $X = \varprojlim X_k$, and any fixed point of $g$ in $X_k$ must map to a fixed point in $X_{k'}$ for $k' < k$. 

We may now dispense easily with the last two claims of the theorem. If $g \in K$ is not $K$-conjugate to any element of $\units{\fr{A}}$, then it clearly cannot be conjugate to any element of $J$ for \emph{any} of the groups $J$ that we are considering. Thus for such an element $g$, the Frobenius formula implies that $\Tr \tau(g) = 0$ for all ramified types $\tau$ of $K$. On the other hand, if $g$ is conjugate to an element of $\units{\fr{A}}$, then for the purpose of computing traces, we may assume $g \in \units{\fr{A}}$, and thus $\Bar{g} \in \GL_n(\resfld_F)$ is upper-triangular. Clearly such a $g$ can have a fixed point in $X_1$ only if all of the diagonal entries of $\Bar{g}$ (its eigenvalues in $\units{\resfld_F}$) are the same, in which case every point of $X_1$ is a fixed point. Thus, if $\Bar{g}$ has at least two distinct eigenvalues, it has no fixed point in $X_1$, and thus has no fixed point in $X_k$ for all $k > 0$. This proves that $\Tr \rho(g) = 0$ for all types in this case. But since the condition on $g$ here depends only on its conjugacy class in $K$, applying the Frobenius formula to $\tau = \Ind_{\units{\fr{A}}}^K(\rho)$ yields $\Tr \tau(g) = 0$ as well. Note that in this case, the trace bound of $n$ in the second claim of the theorem does not arise at all. 

We now deal with the first claim of the theorem. The only remaining possibility for $g$ is that its reduction modulo $\p$ is upper-triangular with one eigenvalue of multiplicity $n$. So, just as in the unramified case, we choose $\alpha \in \units{\roi_F}$ and $h \in \fr{P}^l$ such that $g = \alpha + h$, and such that $l$ is maximal with respect to this decomposition. In order to describe $h$ more explicitly, let $l = nt + r$ with $0 \le r < n$, and for $0 \le i,j < n$ define 
\[ \varepsilon_{ij} = \floor{\frac{n - 1 + r + i - j}{n}} = \begin{cases}
    0 & \text{if } r \le j - i          \\
    1 & \text{if } r - n \le j - i < r  \\
    2 & \text{if } j - i < r - n
\end{cases} . \]
Letting $h_{ij}$ denote the $i,j$ coefficient of the matrix $h$, we may describe $h$ explicitly as follows: \begin{inparaenum}[(i)] \item $\val_F(h_{ij}) \ge t + \varepsilon_{ij}$ for all $i,j$, \item this inequality is an equality for some $i,j$ satisfying $j - i \equiv r \pmod{n}$, and \item if $r = 0$, then the diagonal elements $h_{ii}$ cannot all be the same modulo $\p^{t+1}$. \end{inparaenum} The first of these requirements is precisely the fact that $h \in \fr{P}^l$; the last two are due to the maximality of our choice of $l$. Note that $0 \le n \varepsilon_{ij} + j - i - r < n$ for all $i,j$. So if we let $e_{ij} = n \varepsilon_{ij} + j - i - r$, then $e_{ij}$ is simply the reduction of $j - i - r$ modulo $n$. From this, we get $\val_E(h_{ij}) \ge l + i - j + e_{ij}$, with equality for some $i,j$ such that $e_{ij} = 0$. 

Now let $[u_0 : \varpi u_1 : \dotso : \varpi^{n-1} u_{n-1}]$ represent a point in $X^k$ for some $k$. Then this point is fixed by $g$ if and only if 
\begin{equation} \label{eq:ramified_step1}
    \alpha \varpi^i u_i + \sum_{j=0}^{n-1} h_{ij} \varpi^j u_j = \gamma \varpi^i u_i \pmod{\p_E^{k+i}}, \qquad 0 \le i < n , 
\end{equation}
for some $\gamma \in \units{\roi_E}$. Define a new matrix $a \in M_n(\roi_E)$ by $a_{ij} = \varpi^{j - i - l} h_{ij}$. Then $\val_E(a_{ij}) \ge e_{ij}$, and $a_{ij} \in \units{\roi_E}$ for some $i,j$. The system of equations \eqref{eq:ramified_step1} is now equivalent to 
\begin{equation} \label{eq:ramified_step2}
    \varpi^l \sum_{j=0}^{n-1} a_{ij} u_j = (\gamma - \alpha) u_i \pmod{\p_E^k}, \qquad 0 \le i < n . 
\end{equation}
Now it is immediate that if $k \le l$, then any choice of $u_0, \dotsc, u_{n-1}$ yields a solution to this system (taking $\gamma = \alpha \pmod{\p_E^k}$), and hence every point in $X^k$ is a fixed point of $g$. Assume now that $k > l$. Since $a_{ij} \in \units{\roi_E}$ for some $i,j$ satisfying $j - i \equiv r \pmod{n}$, this system has a solution only if we choose $\gamma = \alpha \pmod{\p_E^l}$ and $\gamma \neq \alpha \pmod{\p_E^{l+1}}$, and thus only if $a_{ij} \in \units{\roi_E}$ for \emph{all} such pairs $i,j$. If the latter condition is is false, then $g$ has no fixed points in $X^k$ and we are finished, so we assume it is true. Let $\bo{u}$ denote the column vector in $\roi_E^n$ having $u_0, \dotsc, u_{n-1}$ as its components. Also, as in the unramified case, let $\gamma' = \varpi^{-l} (\gamma - \alpha)$. Then \eqref{eq:ramified_step2} is equivalent to 
\begin{equation} \label{eq:ramified_step3}
    a \bo{u} = \gamma' \bo{u} \pmod{\p_E^{k-l}} . 
\end{equation}
We now have two cases to consider. First, if $r = 0$, then $\Bar{a} \in \GL_n(\resfld_E)$ is diagonal, but not scalar (since $l$ was chosen to be maximal). Thus, in this case, there can be no fixed points in $X^{l+1}$, and hence none in $X^k$ for all $k > l$. On the other hand, if $r > 0$, then it is easy to see (since, for example, the matrix $a$ has exactly one unit in each row and column) that there will be exactly one solution to \eqref{eq:ramified_step3} for every root of the characteristic polynomial of $a$ in $\rquot{\roi_E}{\fr{p}_E^{k-l}}$. Since the number of such roots is bounded as $k \to \infty$, the number of solutions to \eqref{eq:ramified_step3} is bounded. The fixed points of $g$ in $X^k$ are given by the lifts of these solutions from $\rquot{\roi_E}{\p_E^{k-l}}$ to $\rquot{\roi_E}{\p_E^k}$, and thus are bounded as well. This completes the proof in the ramified case. 

Note that in the last case above, the characteristic polynomial of $\Bar{a}$ is just $x^n - \Bar{\eta}$, where $\eta = \prod a_{ij}$, the product being taken over all $i,j$ such that $j - i \equiv r \pmod{n}$. So we may summarize all of the conditions above as follows: If $r = 0$ or if $\Bar{\eta}$ is not an $n$th power in $\units{\resfld_E}$, there will be no fixed points in $X^k$ for all $k > l$. On the other hand, if $r > 0$ and $\Bar{\eta}$ is an $n$th power in $\units{\resfld_E}$, then Hensel's lemma again yields (except possibly when $n$ is equal to the residual characteristic) that there are at most $n$ roots in $\rquot{\roi_E}{\p_E^k}$ for all $k$. Thus, in this case, there are at most $n (\card{\resfld_E})^{ln}$ fixed points of $g$ in $X^k$ for all $k$, just as in the unramified case. Therefore, in this case, $\Tr(\tau(g))$ is bounded by that number times the maximum dimension of $\lambda$ (which is $q^{\frac{n^2-n}{2}}$) times $[K:\units{\fr{A}}] = \prod_{k=1}^{n-1} (1+q+\dotsb+q^k)$. 
\end{proof}

\section{The archimedean places}
\label{sec:archimedean}

In order to deal with the components of our global types and automorphic representations at the infinite places, we now briefly detour to review the relevant representation theory and set up the necessary notation. Recall that we have fixed, for each infinite place $v$ of $F$, an isomorphism $\iota_v: G(F_v) \to \ro{U}(n)$. Thus if $\tau$ is a global type, $\tau_v = \tau'_v \circ \iota_v$ for some irreducible representation $\tau'_v$ of the compact Lie group $\ro{U}(n)$. We now let $\fr{g} = \fr{u}(n)$ be the (real) Lie algebra of $\ro{U}(n)$, which is the algebra of skew-Hermitian matrices in $M_n(\CC)$. Let $\fr{g}_{\RR} = i \fr{g}$, the algebra of Hermitian matrices in $M_n(\CC)$, and let 
\[ \fr{g}^{\CC} = \fr{g} \otimes_{\RR} \CC = \fr{g}_{\RR} \oplus i \fr{g}_{\RR} = \fr{gl}(n, \CC) . \]
Let $T$ be the maximal torus in $\ro{U}(n)$ consisting of diagonal matrices, and let $\fr{h}$ be the corresponding Cartan subalgebra of $\fr{u}(n)$: 
\[ \fr{h} = \left\{ \begin{pmatrix} ia_1 & & \\ & \ddots & \\ & & ia_n \end{pmatrix} \suchthat a_i \in \RR \right\} . \]
As usual, let $\fr{h}_{\RR} = i \fr{h}$, and let $\fr{h}^{\CC} = \fr{h} \otimes_{\RR} \CC = \fr{h}_{\RR} \oplus i \fr{h}_{\RR}$. Finally, let $\fr{h}_{\RR}^*$ and $(\fr{h}^{\CC})^*$ denote the dual spaces of $\fr{h}_{\RR}$ and $\fr{h}^{\CC}$, respectively. 

To simplify things, we will work relative to a fixed basis: Let $e_i$ be the $n \times n$ matrix with a $1$ in the $i,i$ position and zeros elsewhere, so that $\left\{ e_i \suchthat 1 \le i \le n \right\}$ is a $\CC$-basis of $\fr{h}^{\CC}$. We also let $e_i^*$ denote the functionals of the corresponding dual basis, so $e_i^*(e_j) = \delta_{ij}$ for each $i,j$. Thus any linear functional $\lambda \in (\fr{h}^{\CC})^*$ can be written uniquely as $\lambda = \sum_{i=1}^n a_i e_i^*$ ($a_i \in \CC$), and such a $\lambda$ will be analytically integral if and only if $a_i \in \ZZ$ for all $i$. 
In this setting, the set of roots $\Delta$ of $\ro{U}(n)$ is 
\[ \Delta = \left\{ \lambda_{ij} = e_i^* - e_j^* \suchthat 1 \le i \neq j \le n \right\} . \]
With respect to our chosen basis of $(\fr{h}^{\CC})^*$, the sets of positive and simple roots are, respectively, 
\begin{align*}
    \Delta^+ &= \left\{ \lambda_{ij} \suchthat 1 \le i < j \le n \right\} \text{ and }          \\
    \Pi      &= \left\{ \lambda_{i,i+1} = e_i^* - e_{i+1}^* \suchthat 1 \le i < n \right\} . 
\end{align*}
With these choices, we find that a weight $\lambda = \sum a_i e_i^*$ is dominant if and only if $a_i \ge a_j$ for all $i < j$. By the theorem of the highest weight, the irreducible representations of $\ro{U}(n)$ are in one-to-one correspondence with the set $\Lambda$ of dominant, analytically integral functionals on $\fr{h}^{\CC}$: 
\[ \Lambda = \left\{ \lambda = \sum_{i=1}^n a_i e_i^* \suchthat a_i \in \ZZ \quad \forall i \text{, and } a_1 \ge \dotsb \ge a_n \right\} . \]
For $\lambda \in \Lambda$, we will denote by $\xi_{\lambda}$ the corresponding representation of $\ro{U}(n)$. 

Following standard practice, we define a bilinear form $B_0: \fr{g} \times \fr{g} \to \RR$ by 
\[ B_0(X, Y) = \Tr XY . \]
Note that for any $\lambda \in \fr{h}_{\RR}^*$, there is a unique $H_{\lambda} \in H$ such that $\lambda(H) = B_0(H, H_{\lambda})$ for all $H \in \fr{h}_{\RR}$. We may now define an inner product on $\fr{h}_{\RR}^*$ by 
\[ \langle \lambda_1, \lambda_2 \rangle = B_0(H_{\lambda_1}, H_{\lambda_2}) . \]
Let $\delta$ be half the sum of the positive roots: 
\[ \delta = (\tfrac{n-1}{2})e_1^* + (\tfrac{n-3}{2})e_2^* + \dotsb + (\tfrac{3-n}{2})e_{n-1}^* + (\tfrac{1-n}{2})e_n^* . \]
The Weyl dimension formula now gives 
\begin{align*}
    \dim(\xi_{\lambda}) &= \prod_{\alpha \in \Delta^+} \frac{\langle \lambda + \delta, \alpha \rangle}{\langle \delta, \alpha \rangle} \\
                        &= \prod_{1 \le i < j \le n} \frac{a_i - a_j + j-i}{j-i} \\
                        &= \frac{\prod\limits_{i < j} (a_i - a_j + j-i)}{\prod\limits_{k=1}^{n-1} k!}
\end{align*}
for any $\lambda = \sum a_i e_i^* \in \Lambda$. Note that the above expression is a polynomial of degree $\frac{n^2 - n}{2}$ in the $n$ variables $a_1, \dotsc, a_n$. We will refer to this polynomial as the Weyl polynomial for $\ro{U}(n)$. 

In order to prove our main global theorem, we will need a bound on the characters of the representations $\xi_{\lambda}$, in analogy with Theorem~\ref{th:tracebound}. The following proposition is adapted slightly from~\cite{CC:NumberFields}*{Prop. 1.9}, and the proof may be found there. 

\begin{proposition}[Chenevier-Clozel] \label{prop:ccbound}
Let $g \in \ro{U}(n)$, and assume $g$ is not central. There exists a polynomial in $n$ variables $P_g(X_1, \dotsc, X_n)$, of degree strictly less than that of the Weyl polynomial, such that for all $\lambda = \sum a_i e_i^* \in \Lambda$, 
\[ \abs{\Tr \xi_{\lambda}(g)} \le P_g(a_1, \dotsc, a_n) . \]
\end{proposition}

It will be convenient to abuse notation slightly and refer to the Weyl polynomial and the polynomial $P_g$ above as polynomials on $\fr{h}_{\RR}^*$, with the understanding that when $\lambda = \sum a_i e_i^*$, $P(\lambda)$ means $P(a_1, \dotsc, a_n)$. Note that the degree of such a polynomial is well-defined independently of our choice of basis for $\fr{h}_{\RR}^*$.

\section{Proof of main theorem}
\label{sec:maintheoremproof}

We now return to the global setting. In all that follows, our notation will be as in section~\ref{sec:globalintro}. In particular, $F$ is now a totally real number field, $G$ a unitary group of rank $n$ defined over $F$, and $K$ a certain maximal compact open subgroup of $G(\AA)$ on which our global types are defined. 

We first record a crucial lemma, which is an immediate consequence of Theorem~\ref{th:tracebound} and Proposition~\ref{prop:ccbound}. 

\begin{lemma} \label{lem:globalbound}
Let $x \in K \smallsetminus Z(\AA)$, and assume $x$ is semisimple. Then there exists a constant $C_x$, and for each infinite place $v$ of $F$ a polynomial $P_{x,v}$ on $\fr{h}_{\RR}^*$, such that for all global types $\tau$, 
\[ \abs{\Tr \tau(x)} \le C_x \cdot n^{\card{S(\tau)}} \cdot \prod_{v \divs \infty}P_{x,v}(\lambda_v(\tau)) . \]
Each of the polynomials $P_{x,v}$ has degree strictly less than that of the Weyl polynomial of $\ro{U}(n)$. 
\end{lemma}
\begin{proof}
Since $x$ is semisimple and not in the center, it must have at least two distinct eigenvalues. Thus there are at most finitely many places $v \in S$ at which the reduction of $x_v$ modulo $\p_{F_v}$ has a single eigenvalue of multiplicity $n$. For each of these places, Theorem~\ref{th:tracebound} gives us a constant $C_{x_v}$ such that 
\[ \abs{\Tr \tau_v(x_v)} \le C_{x_v} \]
for every supercuspidal $K$-type $\tau_v$, where $K = \GL_n(\roi_{F_v})$. Let $C_x$ be the product of these constants $C_{x_v}$. At all other finite places $v \in S$, we have by the same theorem 
\[ \abs{\Tr \tau_v(x_v)} \le n \]
for every supercuspidal $K$-type $\tau_v$. For each infinite place, let $P_{x,v}$ be the polynomial given by Proposition~\ref{prop:ccbound} applied to $\iota_v(x_v) \in \ro{U}(n)$. Then for any global type $\tau = \bigotimes_v \tau_v$, since $\tau_v$ is $1$-dimensional outside of $\infty$ and $S(\tau)$, we have 
\[ \abs{\Tr \tau(x)} = \prod_{v \in S(\tau)} \abs{\Tr \tau_v(x_v)} \cdot \prod_{v \divs \infty} \abs{\Tr \tau_v(x_v)} , \]
and the result follows. 
\end{proof}

\begin{proof}[Proof of Theorem~\ref{th:globaltheorem}]
We now proceed to compute $m(\tau)$. As mentioned in section~\ref{sec:globalintro}, $\ca{A}(\lquot{G(F)}{G(\AA)})$ is simply the space of smooth functions on $G(\AA)$ that are invariant under left translation by elements of $G(F)$. Note that this is merely the induced representation: 
\begin{equation} \label{eq:autspace}
    \ca{A}(\lquot{G(F)}{G(\AA)}) = \Ind_{G(F)}^{G(\AA)} (1) . 
\end{equation}
To deal with the restriction of this representation to $K$, we apply Mackey's formula. Let $R$ be a set of double coset representatives for 
\[ \biquot{G(F)}{G(\AA)}{K} . \]
Note that $R$ is finite, for example by~\cite{Shimura}*{8.7}. To simplify notation, we let $K_g = G(F)^g \inter K$ for $g \in R$. Since $K$ is a compact subgroup of $G(\AA)$ and $G(F)^g$ a discrete subgroup, $K_g$ is finite. Applying Mackey's formula to \eqref{eq:autspace} now yields 
\[ \Res^{G(\AA)}_K \Ind_{G(F)}^{G(\AA)} (1) = \bigoplus_{g \in R} \Ind_{K_g}^K \Res^{G(F)^g}_{K_g} (1) . \]
Now if $\tau$ is any global type for $G$, then $m(\tau)$ is merely the multiplicity of $\tau$ in the representation above. So we have (relaxing our notation somewhat, as the restriction functors are implied) 
\begin{align*}
    m(\tau) &= \left\langle \tau, \bigoplus_{g \in R} \Ind_{K_g}^K (1) \right\rangle_K  \\
            &= \sum_{g \in R} \left\langle \tau, \Ind_{K_g}^K (1) \right\rangle_K       \\
            &= \sum_{g \in R} \left\langle \tau, 1 \right\rangle_{K_g}                  \\
            &= \sum_{g \in R} \frac{1}{\card{K_g}} \sum_{x \in K_g} \Tr \tau(x) 
\end{align*}
for all global types $\tau$. Note that $K_g \inter Z(\AA) = G(F) \inter Z_0 = \roi_E^1 = \mu_E$, and by definition a global type $\tau$ is assumed to be trivial on this subgroup. Thus in the last sum above, the terms for which $x$ is central all satisfy $\Tr \tau(x) = \dim(\tau)$. Letting $C_1 = \card{\mu_E} \cdot \left(\sum_{g \in R} \frac{1}{\card{K_g}} \right)$, we have, for all global types $\tau$, 
\[ m(\tau) = C_1 \dim(\tau) + \sum_{g \in R} \frac{1}{\card{K_g}} \sum_{\substack{x \in K_g \\ x \notin Z(\AA)}} \Tr \tau(x) . \]
Now let $g \in R$, and let $x \in K_g \smallsetminus Z(\AA)$. Then $x$ is of finite order (as it belongs to the finite group $K_g$) and thus is semisimple. Hence we can apply the lemma to $x$, to get a constant $C_x$ and polynomials $P_{x,v}$ for $v \divs \infty$, such that 
\[ \abs{\Tr \tau(x)} \le C_x \cdot n^{\card{S(\tau)}} \cdot \prod_{v \divs \infty}P_{x,v}(\lambda_v(\tau)) \]
for all global types $\tau$. As there are only finitely many such $x$ to consider, we may sum the constants $C_x$ and the polynomials $P_{x,v}$, and the result follows. 
\end{proof}

\bibliographystyle{plain}
% \bibliography{references}

\input{unitarygroups.bbl}
% \begin{bibdiv}
% \begin{biblist}
% \bibselect{references}
% \end{biblist}
% \end{bibdiv}

\end{document}

%% file: common.tex
\usepackage{amsmath}
\usepackage{amssymb}
\usepackage{amsfonts}
\usepackage{amsthm}
\usepackage{paralist}   % Provides inparaenum
\usepackage{multirow}
\usepackage{booktabs}   % Provides nice looking hrules for tables

\usepackage[pdftex, colorlinks=false, pdfborder={0 0 0}]{hyperref}

\theoremstyle{plain}
\newtheorem{theorem}{Theorem}[section]
\newtheorem*{theorem*}{Theorem}
\newtheorem{proposition}[theorem]{Proposition}
\newtheorem*{proposition*}{Proposition}
\newtheorem{lemma}[theorem]{Lemma}
\newtheorem*{lemma*}{Lemma}
\newtheorem{corollary}[theorem]{Corollary}
\newtheorem*{corollary*}{Corollary}

\newtheorem*{conjecture*}{Conjecture}

\newtheorem*{hypothesis*}{Hypothesis}

\theoremstyle{definition}
\newtheorem{definition}[theorem]{Definition}
\newtheorem*{definition*}{Definition}

\newtheorem*{example*}{Example}

\theoremstyle{remark}

\newtheorem*{remark*}{Remark}
\newtheorem*{note*}{Note}
\newtheorem*{claim*}{Claim}

\numberwithin{equation}{section}

\newcommand{\ie}{i.e., }    % id est
\newcommand{\etc}{etc.}     % et cetera - May need to use '\ ' after
\newcommand{\define}[1]{\emph{#1}}  % For informal inline definitions

\DeclareMathOperator{\End}{End}

\DeclareMathOperator{\Ker}{Ker}

\DeclareMathOperator{\Gal}{Gal}
\DeclareMathOperator{\Tr}{Tr}

\DeclareMathOperator{\Ind}{Ind}

\DeclareMathOperator{\Res}{Res}

\DeclareMathOperator{\GL}{GL}

\DeclareMathOperator{\val}{val}

\newcommand{\ZZ}{\mathbb{Z}}    % The integers
\newcommand{\RR}{\mathbb{R}}    % The reals
\newcommand{\CC}{\mathbb{C}}    % The complex numbers
\newcommand{\FF}{\mathbb{F}}    % A finite field
\newcommand{\PP}{\mathbb{P}}    % Projective space
\renewcommand{\AA}{\mathbb{A}}  % The adeles (of a number field)
\newcommand{\fr}[1]{\mathfrak{#1}}  % Fraktur symbols
\newcommand{\ca}[1]{\mathcal{#1}}   % Calligraphic symbols
\newcommand{\bo}[1]{\mathbf{#1}}    % Bold-face symbols
\newcommand{\ro}[1]{\mathrm{#1}}    % Roman typeface

\newcommand{\roi}{\mathfrak{o}}             % The ring of integers in a nonarch local field
\newcommand{\p}{\mathfrak{p}}               % The prime ideal in a nonarch local field
\newcommand{\resfld}{\boldsymbol{k}}        % The residue field of a nonarch local field
\newcommand{\units}[1]{#1^{\times}}         % The units of a ring, read "R cross"
\newcommand{\inter}{\cap}                               % intersection
\newcommand{\normsubgp}{\lhd}                           % normal subgroup
\newcommand{\divs}{\mid}                                % divides
\newcommand{\suchthat}{\ \middle\vert\ }                % such that
\newcommand{\card}[1]{\# #1}                            % cardinality
\newcommand{\abs}[1]{\left\lvert #1 \right\rvert}       % absolute value (single vertical bars)
\newcommand{\floor}[1]{\left\lfloor #1 \right\rfloor}   % floor
\newcommand{\ceil}[1]{\left\lceil #1 \right\rceil}      % ceiling

\renewcommand{\Bar}[1]{\overline{#1}}

\newcommand{\rquot}[2]{\left. #1 \middle/ #2 \right.}
\newcommand{\lquot}[2]{\left. #1 \middle\backslash #2 \right.}
\newcommand{\biquot}[3]{\left. #1 \middle\backslash #2 \middle/ #3 \right.}